\documentclass[final,nomarks]{dmtcs-episciences} 
\usepackage[T1]{fontenc}
\usepackage[utf8]{inputenc}

\usepackage{amsfonts,amsthm,amsmath,amssymb}
\usepackage{microtype,hyperref}
\usepackage[capitalize,nameinlink]{cleveref}

\usepackage[round]{natbib}

\newtheorem{theorem}{Theorem}

\newtheorem{lemma}{Lemma}
\newtheorem{corollary}[lemma]{Corollary}
\newtheorem{observation}[lemma]{Observation}

\theoremstyle{definition}
\newtheorem{definition}[lemma]{Definition}

\title[Widths of Strict Outerconfluent Graphs]{The Widths of Strict Outerconfluent Graphs}

\author[David Eppstein]{David Eppstein\affiliationmark{1}\thanks{This research was supported in part by NSF grant CCF-2212129.}}
\affiliation{Department of Computer Science, University of California, Irvine, USA}


\publicationdata{vol. 26:3}{2024}{20}{10.46298/dmtcs.12862}{2024-01-12; None}{2024-11-13}

\begin{document}
\maketitle  

\begin{abstract}
Strict outerconfluent drawing is a style of graph drawing in which vertices are drawn on the boundary of a disk, adjacencies are indicated by the existence of smooth curves through a system of tracks within the disk, and no two adjacent vertices are connected by more than one of these smooth tracks. We investigate graph width parameters on the graphs that have drawings in this style. We prove that the clique-width of these graphs is unbounded, but their twin-width is bounded.
\end{abstract}

\section{Introduction}

\emph{Confluent drawing} is a powerful style of graph drawing that permits many non-planar and dense graphs to be drawn without crossings \citep{DicEppGoo-JGAA-05,EppGooMen-GD-05,HirMeiRap-GD-06,HuiPelSch-Algo-07,EppGooMen-Algo-07,QueAnc-GD-10,CorDia-GD-22}. A confluent drawing consists of a system of non-crossing smooth curves in the plane, called \emph{tracks}, whose endpoints are either vertices of the graph or \emph{junctions} where several tracks meet, all having the same slope at that point. Two vertices are adjacent whenever  the union of some of the tracks forms a smooth curve connecting them. In this way, each confluent drawing represents unambiguously a unique graph, unlike the \emph{bundled drawings} \citep{LhuHurTel-CGF-17} which they otherwise resemble. Applications of confluent drawing include the automated layout of syntax diagrams \citep{BanBroEpp-GD-15}, and the simplification of the Hasse diagrams of partially ordered sets \citep{EppSim-JGAA-13}. A constrained version of confluent drawing, called \emph{strict confluent drawing}, requires that each adjacency be represented by only one smooth curve \citep{EppHolLof-JoCG-16,ForGanKlu-JGAA-21}. In \emph{outerconfluent drawings}, the tracks are interior to a disk whose boundary contains the vertices. In this work we study \emph{strict outerconfluent graphs}, the graphs that have strict outerconfluent drawings.\footnote{For the full definition of strict outerconfluent graphs, see \cref{def:outerconf}.} If the vertex ordering along the drawing boundary is given, these graphs may be recognized in polynomial time \citep{EppHolLof-JoCG-16}, but their recognition without this information, and other algorithmic problems concerning them, remain mysterious.

In this work, following \citet{ForGanKlu-JGAA-21}, we study the width of strict outerconfluent graphs. There are many graph width parameters, of which treewidth is perhaps the most famous. Treewidth is bounded for some types of graph drawing with vertices on the boundary of a disk (outerplanar and outer-$k$-planar drawings \citep{WooTel-NYJM-07}), suggesting that, analogously, strict outerconfluent graphs might have bounded width of some sort. However, graphs of bounded treewidth are sparse, and strict outerconfluent graphs can be dense: for instance they include the complete graphs and complete bipartite graphs. Therefore, a different concept of width is needed, one that can be bounded for dense graphs. Among these widths, we focus on two, \emph{clique-width} and \emph{twin-width}.\footnote{For definitions of these two width parameters, see \cref{def:clique-width} and \cref{def:twin-width}.}
 
For sparse graphs, clique-width is equivalent to treewidth, in the sense that if one of these two width parameters is bounded, the other one is also bounded \citep{GurWan-WG-00}, but graphs of bounded clique-width can also be dense.
The strict outerconfluent graphs include the distance-hereditary graphs, which are known to have bounded clique-width \citep{EppGooMen-GD-05}. \citet{ForGanKlu-JGAA-21} defined a sub-class of strict outerconfluent drawings, the \emph{tree-like} outerconfluent drawings, in which the tracks that are incident to junctions must form a single topological tree within the drawing, and proved that their graphs also have bounded clique-width \citep{ForGanKlu-JGAA-21}. We prove that, in contrast, there exist strict outerconfluent graphs with unbounded clique-width.

To complement this result, we prove that another width parameter of these graphs, their \emph{twin-width}, is bounded. Twin-width is bounded for many classes of graphs of interest in graph drawing, including the planar and $k$-planar graphs, and the graphs of bounded genus. It is also bounded for graphs of bounded clique-width \citep{TW2}. The algorithmic consequences of bounded twin-width include the existence of a fixed-parameter tractable algorithm for testing whether a given graph models a given formula of first-order logic, parameterized by the size of the formula (including as a special case subgraph isomorphism) \citep{TW1}, and better approximation algorithms for dominating set, independent set, and graph coloring than the best approximations known for more general families \citep{TW3,BerBonDep-STACS-23}.

We prove the following results:
\begin{itemize}
\item The strict outerconfluent graphs do not have bounded clique-width (\cref{thm:clique-width}).
\item The strict outerconfluent graphs have bounded twin-width. A twin-width decomposition of bounded width can be constructed for these graphs in polynomial time, given their vertex ordering around the boundary of a strict outerconfluent drawing.
\end{itemize}

The main idea of the first result is to find a recursive construction of a family of strict outerconfluent drawings for which we can prove unbounded \emph{rank-width}, a graph width parameter closely related to clique-width. The main idea of the second result is to harness known results relating the growth rate of a family of \emph{ordered graphs} (pairs of a graph and a linear ordering on its vertices) to the twin-width of the family, and to use the fact that strict confluent drawings have only linearly many junctions \citep{EppHolLof-JoCG-16} to show that they have a small growth rate.

\section{Definitions}

For completeness we repeat the following definitions, from previous work, of the main concepts considered in our results. We assume familiarity with the basic concepts of graph theory and of two-dimensional topology. By a \emph{graph} we always mean a finite undirected graph, without multiple adjacencies or loops.

\begin{definition}
\label{def:outerconf}
A \emph{strict outerconfluent drawing} of a graph $G$ consists of a system of finitely many smooth curves in a topological disk, which we call \emph{tracks},\footnote{In some past work on confluent drawings these curves have been called \emph{arcs}, but that terminology conflicts with standard graph-theoretic terminology for directed edges.} disjoint except for shared endpoints. These endpoints have two types: some are identified one-for-one with vertices of $G$, while others are called \emph{junctions}. Each vertex must lie on the boundary of the disk. At a junction, three or more tracks must meet, all having the same slope. A smooth curve within the union of tracks, starting and ending at vertices and otherwise passing only through tracks and junctions, is called an \emph{edge curve}. Each two adjacent vertices of $G$ must be the endpoints of a unique edge curve. No edge curve may connect a vertex to itself, or connect non-adjacent vertices. Each track must be part of at least one edge curve. A \emph{strict outerconfluent graph} is a graph that has a strict outerconfluent drawing.
\end{definition}

\begin{definition}
\label{def:clique-width}
The \emph{clique-width} of an undirected graph is the minimum number of colors needed to construct the graph by a sequence of the following four operations on (improperly) colored graphs:
\begin{itemize}
\item Create a single-vertex graph, with its vertex given any of the available colors.
\item Take the disjoint union of two colored graphs.
\item Recolor all vertices of one color to another color (possibly one that is already used by other vertices).
\item Perform a \emph{color join} operation that adds edges between all pairs of vertices of two specified colors. 
\end{itemize}
\end{definition}

\begin{definition}
\label{def:twin-width}
\emph{Twin-width} is defined through a type of graph decomposition in which clusters of vertices are merged in pairs, starting with one cluster per vertex, until only one cluster is left. At each step of the decomposition, two clusters are connected by a \emph{red edge} if some but not all adjacencies exist between vertices of one cluster and vertices of the other. The goal is to find a decomposition sequence that minimizes the maximum degree of the resulting sequence of red graphs. The twin-width of a graph $G$ is the minimum value of $d$ such that there exists a decomposition of $G$ for which, after each pairwise merge, the red graph has maximum degree at most $d$ \citep{TW1}.
\end{definition}

\section{Unbounded clique-width}

\begin{figure}[t]
\includegraphics[width=\textwidth]{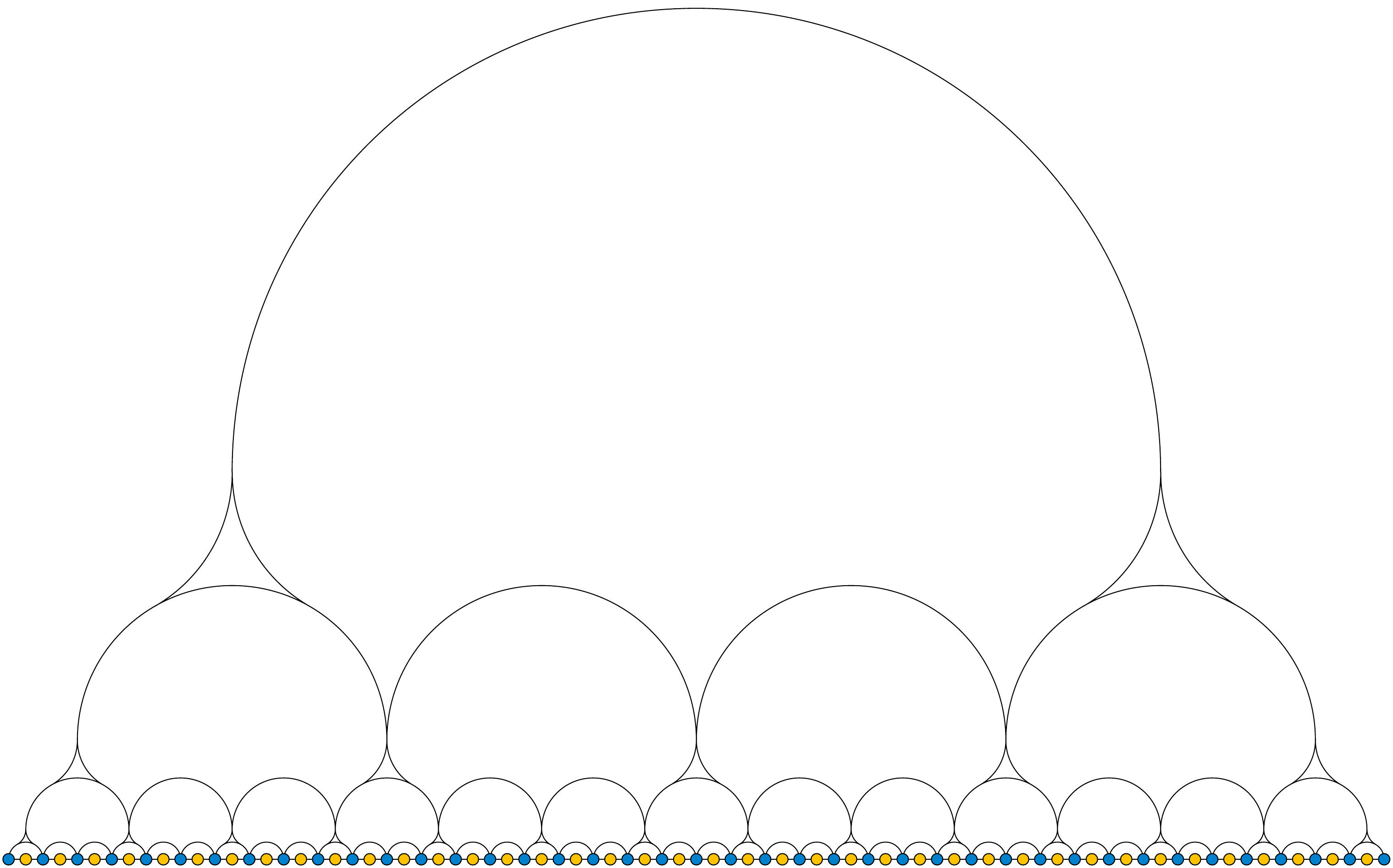}
\caption{Recursively constructed non-tree-like strict outerconfluent graph $G_4$}
\label{fig:non-tree-like}
\end{figure}

In this section we prove that strict outerconfluent graphs can have unbounded clique-width. We were unable to prove this using standard examples of graphs with high clique-width, such as grid graphs; indeed, even a $3\times 3$ grid is not strict outerconfluent. (The outer induced 8-cycle has only one strict outerconfluent drawing, because it has no 4-cycles to provide non-trivial confluence, and it is not possible to add the central vertex of the grid to this drawing.) Instead, our proof is based on a family of drawings depicted in \cref{fig:non-tree-like}, which we construct as follows:

\begin{definition}
Let $G_k$ be the graph represented by a confluent drawing constructed as follows.
\begin{itemize}
\item It is convenient to shape the disk on which the graph is drawn as a half-plane above a horizontal bounding line, to match the depiction in the figure. (This is merely a convention for describing the drawing and does not affect its combinatorial structure.)
\item On the boundary line of the half-plane, place $3^k$ vertices (the alternating blue and yellow vertices of the figure), connected by tracks that directly connect consecutive pairs of vertices (drawn along the boundary line). These are the only tracks incident to the $\lfloor 3^k/2\rfloor$ yellow vertices. Additional tracks will extend vertically from the $\lceil 3^k/2\rceil$ blue vertices.
\item The remaining tracks of the drawing are arranged into $k$ levels, each of which is drawn within a slab of the half-plane bounded between two horizontal lines. Number these levels from $0$ to $k-1$, bottom to top. The bottom line of the $i$th level contains $\lceil 3^{k-i}/2 \rceil$ points (vertices on level $0$, junctions at higher levels) at which tracks extend with a vertical tangent into that level; number these points as $p_{i,j}$ with $0\le j<3^{k-i}/2$.
\item Within level $i$, connect each two consecutive points $p_{i,j}$ and $p_{i,j+1}$ by a semicircle. If $j$ is a multiple of three, subdivide this semicircle by two junctions into three circular arc tracks; for other values of $j$, this semicircle is itself a track. As a special case, for the top level (level $k-1$) the single semicircle connecting points $p_{i,0}$ and $p_{i,1}$ is not subdivided, and forms a track. In the figure, the arcs into which the semicircles are subdivided span angles of~$\pi/3$.
\item For each level $i$ except the top level, and each subdivided semicircle connecting points $p_{i,j}$ and $p_{i,j+1}$ where $j$ is a multiple of three, add tracks connecting the two subdivision points to the point $p_{i+1,j/3}$ on the upper boundary line of the level. At the two junctions on the semicircle, these tracks should be oriented so that each one connects $p_{i+1,j/3}$ downward by a smooth curve through the semicircle to the two points $p_{i,j}$ and $p_{i,j+1}$. In the figure, these upward tracks are also arcs of circles, congruent to the arcs of the subdivided semicircle.
\end{itemize}
\end{definition}

For instance, the figure depicts $G_4$. By construction, $G_k$ has exactly $n=3^k$ vertices.

\begin{observation}
\label{obs:track-neighbors}
Any semicircular track at level $i$ of $G_k$ has smooth paths connecting it to $2^i$ vertices, $2^{i-1}$ on its left and $2^{i-1}$ on its right. The track is used by edges of $G_k$ that connect each of the vertices in the left subset to each of the vertices in the right subset. 
These two subsets are separated by a gap of $3^{i-1}$ vertices, wide enough that it cannot be spanned by any semicircular track at a lower level of $G_k$. 
\end{observation}

It follows that $G_k$ has $\Theta(4^k)$ edges, enough to make it not sparse. More precisely, the number of edges can be calculated as
\[
\frac{8\cdot 4^k - 3\cdot 3^k - 5}{6}.
\]
We omit the details as this calculation is not important for our results.

\begin{lemma}
$G_k$ is strict outerconfluent.
\end{lemma}

\begin{proof}
Each smooth curve from vertex to vertex must go upward through the levels of the track, follow a single semicircular track at some level, and then go back downwards through the levels, because there are no tracks that smoothly connect downward-going curves to upward-going curves. A smooth curve from vertex to vertex that uses a semicircular track at level $i$ must connect two vertices that are at least $3^{i-1}+1$ steps apart and at most $3^i-1$ steps apart. Because these numbers of steps form disjoint ranges for disjoint levels, no two curves using semicircles from different levels can connect the same two vertices. Two semicircular tracks at the same level that do not share a confluent junction have disjoint subsets of vertices that they can reach. Two semicircular tracks at the same level that do share a confluent junction cannot provide two paths between any pair of vertices, because one of the tracks connects vertices that can reach the shared junction to other vertices to the left of the junction, while the other track connects only to the right.
\end{proof}

Rather than working directly with clique-width, it is convenient to use \emph{rank-width}, a closely related quantity derived from hierarchical clusterings of the vertices of a given graph.

\begin{definition}
\label{def:rank-width}
Define a \emph{hierarchical clustering} of a graph to be a ternary tree having the graph's vertices as its leaves. For each edge $e$ of such a tree, removing $e$ from the tree partitions it into two subtrees, and thus defines a partition of the vertices into two subsets; call this partition the \emph{cut} associated with $e$, and call the two subsets the \emph{sides} of the cut. For any of these cuts, we can form a binary \emph{biadjacency matrix} whose rows correspond to the vertices on one side of the cut, and whose columns correspond to the vertices on the other side (choosing arbitrarily which side to use for which role). The coefficient of this matrix in a given row and column is one if the corresponding two vertices are adjacent, and zero otherwise. (For the purposes of defining rank-width, these coefficients are defined within the finite field $\mathbb{Z}_2$, rather than as real numbers, but that makes little difference for our purposes.) The rank-width of the graph is the maximum rank of any of the biadjacency matrices of these cuts, for a hierarchical clustering chosen to minimize this maximum rank.
\end{definition}

\begin{lemma}[\citet{OumSey-JCTB-06}]
Let $G$ be any graph, let $r$ be its rank-width and let $c$ be its clique-width. Then
\[ r\le c\le 2^{r+1}-1. \]
\end{lemma}

Thus, the rank-width of a family of graphs is bounded if and only if the clique-width is bounded.

\begin{definition}
A \emph{balanced cut} of an $n$-vertex graph is a partition of its vertices into two subsets that each have at least $n/3$ vertices.
\end{definition}

\begin{lemma}
\label{lem:balanced-cut}
Any graph of rank-width $r$ has a balanced cut whose biadjacency matrix has rank $\le r$.
\end{lemma}

\begin{proof}
Let $G$ be the given graph, and let $n$ be its number of vertices.
Let $T$ be a ternary tree with the vertices of $G$ as its leaves, and with cuts whose biadjacency matrices have rank $\le r$, which exists by the definition of rank-width.  As in \cref{def:rank-width}, define the two sides of an edge $e$ of $T$ to be the two subsets of vertices of $G$ separated in $T$ by $e$. Define a side to be \emph{small} if it consists of fewer than $n/3$ vertices of $G$, and \emph{large} otherwise. If we can find an edge with no small side, it will define a balanced cut, which by construction will have rank $\le r$.

To find an edge of $T$ with no small side, start at any edge of $T$, and then construct a walk as follows. As long as the walk has reached an edge $e$ with a small side, consider the two edges of $T$ that are incident to $e$ on its large side. The large side of $e$ includes $>2n/3$ vertices of $G$ (because the other side is small), so at least one of these two edges, $e'$, separates $e$ from $>n/3$ vertices of $G$. Select $e'$ as the next edge in the walk. The walk terminates when it reaches an edge of $T$ that defines a balanced cut, but it remains to prove that this always happens.

After each step of the walk from an edge $e$ to an edge $e'$, one of the two sides of $e'$ (the side that $e'$ separates from $e$) includes $>n/3$ vertices of $G$, by construction. The other side of $e'$ can be small, but if it is, it is a strict superset of the small side of $e$.  Because the numbers of vertices on the small sides of the edges in this walk form a strictly increasing sequence of integers, they must eventually reach a number that is at least $n/3$, at which point the walk terminates.
\end{proof}

We will prove our result by showing that, for every fixed $r$, $G_k$ has no such low-rank balanced cut, contradicting \cref{lem:balanced-cut}.

\begin{definition}
Given any partition of the vertices of $G_k$ into two subsets, define a \emph{block} of the partition to be a contiguous subsequence of the vertices (as ordered along the boundary line of the drawing of $G_k$) that belongs to one of the two subsets, and is not part of any larger such contiguous subsequence.
\end{definition}

\begin{lemma}
\label{lem:few-subintervals}
If a partition of the vertices of $G_k$ into two subsets has a biadjacency matrix of rank~$\le r$, it has $\le 4r+1$ blocks.
\end{lemma}

\begin{proof}
Each two consecutive blocks contain two consecutive vertices in the ordering of $G_k$, one yellow and one blue in the alternating coloring of $G_k$ from the figure. Assume for a contradiction that there are $\ge 4r+2$ blocks. These would lead to $\ge 4r+1$ yellow--blue edges between consecutive vertices, crossing from block to block and from one side of the partition to the other. Among these, some subsequence of $\ge 2r+1$ edges all have their yellow vertices in the same subset of the partition as each other and their blue vertices in the other subset. Select the edges in odd positions of this subsequence.

This gives a subsequence of $\ge r+1$ edges of $G_k$, connecting vertices that are all consistently colored yellow in one subset of the partition and blue in the other subset. Moreover, because we selected only the edges in odd positions from a longer sequence, no two of these edges have endpoints that are consecutive in the vertex ordering of $G_k$. Because the yellow vertices of $G_k$ are adjacent only to the two consecutive blue vertices, the only yellow--blue edges connecting the selected vertices are the ones in the selected subsequence. That is, this subsequence forms an induced matching within the subgraph of edges of $G_k$ that cross the given partition. The submatrix of the biadjacency matrix, corresponding to this induced matching, is a permutation matrix of rank $\ge r+1$. This contradicts the assumption of the lemma that the whole biadjacency matrix has rank $r$. This contradiction proves that our other assumption, that there are $4r+2$ or more blocks, cannot be true.
\end{proof}

\begin{definition}
As is standard in asymptotic analysis, the ``big Omega notation'' $f(x)=\Omega\bigl(g(x)\bigr)$, where $f(x)$ and $g(x)$ are expressions with a single free variable $x$, indicates that there exist constants $x_0$ and $\alpha>0$ such that, for all $x>x_0$, $f(x)\ge \alpha\,g(x)$. More intuitively, $f$ grows at least proportionally to $g$. We also use $\Omega\bigl(g(x)\bigr)$, separated from the ``$f(x)=$'' part of this notation, to denote a quantity $q(x)$ such that $q(x)=\Omega\bigl(g(x)\bigr)$.
However, the meaning of this notation becomes unclear when there is more than one free variable: even if all variables are assumed to grow unboundedly, there may be different relations between $f$ and $g$ depending on the relative growth rates of the variables. To sidestep these issues we only use $\Omega$-notation for a single free variable. We use the modified notation $\Omega_c$, as in $f(x,c)=\Omega_c\bigl(g(x,c)\bigr)$, to indicate that $c$ is \emph{not} the free variable in the relation between $f$ and $g$. Instead, this notation is a shorthand for $\forall c\, f(x,c)=\Omega\bigl(g(x,c)\bigr)$: for each $c$, there must exist $x_0$ and $\alpha$ (both possibly depending on $c$), such that, for all $x>x_0$, $f(x,c)\ge \alpha\,g(x,c)$.
\end{definition}

\begin{lemma}
\label{lem:close-intervals}
For any constant $c$, and any balanced partition of $G_k$ forming at most $c$ blocks in each subset of the partition,
there exist two blocks on opposite sides of the partition such that the smaller of their two lengths is $\Omega_c(3^{k/c^2})$ times larger than $1+\ell$, where $\ell$ is the number of vertices of $G_k$ that lie between the two blocks in the vertex ordering of $G_k$.
\end{lemma}

\begin{proof}
Label the two subsets of the partition as ``red'' and ``blue''.
Let $R$ be the largest red block, consisting of $\Omega_c(3^k)$ vertices.
Form a sequence of blue blocks $B_1,B_2,\dots$
where $B_1$ is the closest blue block to $R$ and
each successive $B_i$, $i>1$ is the closest blue block to $R$ that is larger than $B_{i-1}$.
Consider the sequence $1,|B_1|,|B_2|,\dots$, where $|B_i|$ is the size of block $B_i$.
This sequence is short (it has length at most $c+1$) and has a large final value $\Omega_c(3^k)$, so it must sometimes increase rapidly: there must be some $i$ for which $|B_i|$ is larger than the previous value in the sequence by a factor of $\Omega_c(3^{k/c})$.

The blue blocks between $R$ and $B_i$ are all short: they are smaller than $B_i$ by a factor of $\Omega_c(3^{k/c})$.
If the red blocks between $R$ and $B_i$ are also all short then $R$ and $B_i$ are close together and they together supply the desired pair of blocks.

Otherwise, form another sequence of red blocks $R_i$ between $B_i$ and $R$, where $R_1$ is the closest red block to $B_i$ in this range,
and each successive $R_i$ is the closest red block to $B_i$ that is larger than $R_{i-1}$, finishing the sequence with $R$ itself. The sequence $|B_i|/3^{k/c}, |R_1|, |R_2|, \dots$ is short, and starts off a factor of at least $3^{k/c}$ smaller than its final value $|R|$, so it must again sometimes increase rapidly: there must be some $j$ for which $R_j$ is larger than the previous value in the sequence by a factor of $\Omega_c(3^{k/c^2})$. All blocks between $R_j$ and $B_i$ are shorter than both $R_j$ and $B_i$ by this factor, so $R_j$ and $B_i$ are close together and they together supply the desired pair of blocks.
\end{proof}

\begin{figure}
\centering\includegraphics[width=0.4\textwidth]{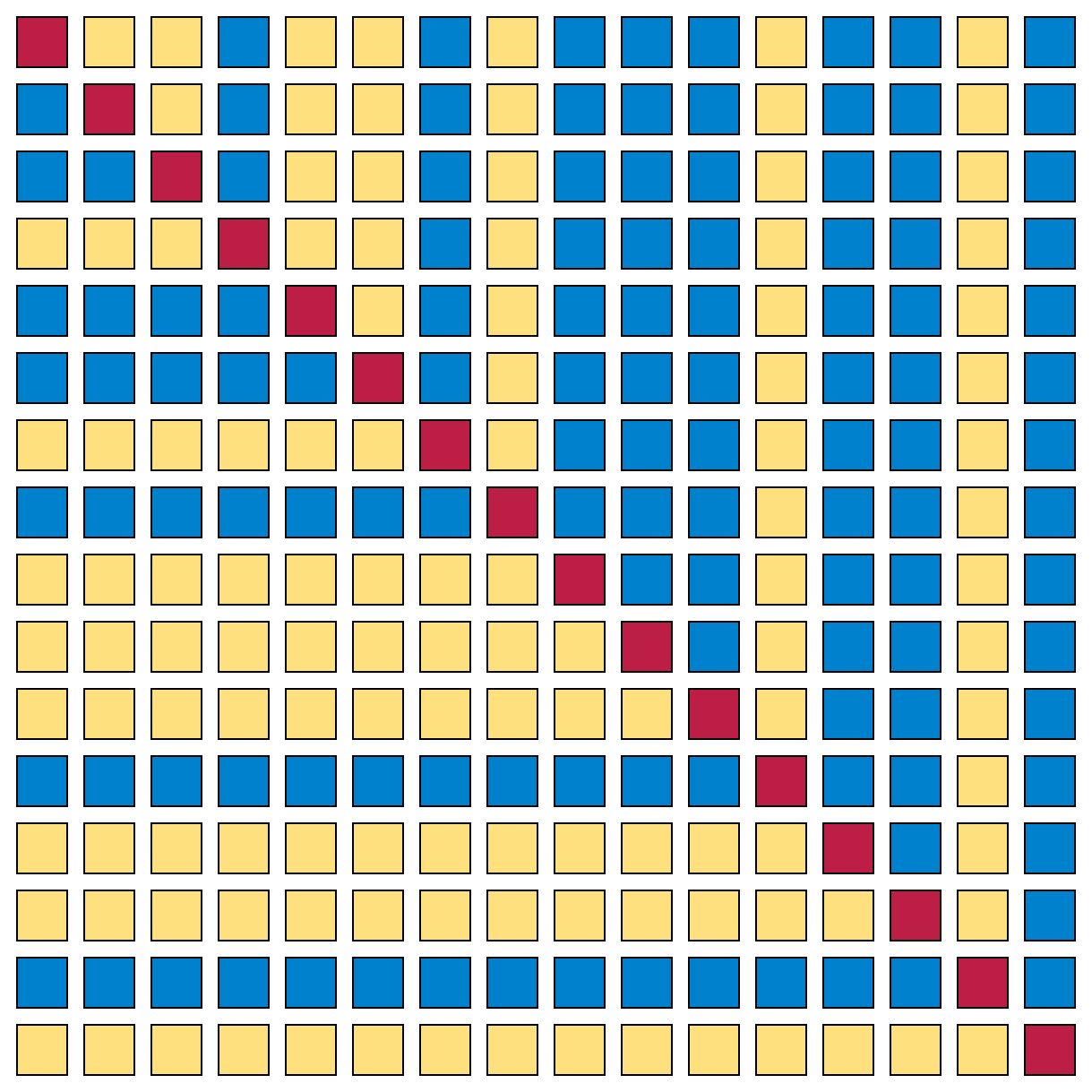}
\caption{Schematic view of a matrix described by \cref{lem:arrow}. The nonzero main diagonal entries are red. Each diagonal entry has zeroes either above it or to the left of it, shown in pale yellow. The remaining blue squares mark entries whose value can be arbitrary.}
\label{fig:arrow}
\end{figure}

\begin{lemma}
\label{lem:arrow}
Let $M$ be a square matrix over any field with the following structure: its diagonal entries are all nonzero, and for each diagonal entry either all entries above it in the same column or all entries to the left of it in the same row are zero. The remaining entries can be arbitrary. (See \cref{fig:arrow}.) Then $M$ has full rank.
\end{lemma}

\begin{proof}
This follows by induction on the size of $M$. The submatrix formed by removing the last row and last column has full rank by induction. If the last column is zero except on the diagonal, the last row does not belong to the row space of the earlier rows; symmetrically, if the last row is zero except on the diagonal, the last column does not belong to the column space of the earlier columns. In either case, including this row or column adds one to the rank.
\end{proof}

\begin{definition}
We say that two semicircular tracks in distinct levels $i<j$ of $G_k$ are \emph{nested} whenever the horizontal interval spanned by the track at level $i$ is a subset of the horizontal interval spanned by the track at level $j$. Equivalently, every vertical line that crosses the track at level $i$ also crosses the track at level~$j$.
\end{definition}

\begin{observation}
\label{obs:nest}
When two tracks are nested, at levels $i<j$, the leftmost and rightmost vertices that can be reached from the track at level $i$ lie strictly between the leftmost and rightmost vertices that can be reached from the track at level $j$.
\end{observation}

\begin{observation}
\label{obs:toofar}
If a set of semicircular tracks is pairwise nested, let $t$ be any track in this sequence and let $L$ and $R$ be the subsets of vertices that can be reached by smooth paths from the left and right sides of $t$ respectively. Then the gap between $L$ and $R$ is longer than the interval spanned by the reachable vertices of any lower-level track (\cref{obs:track-neighbors}). Therefore, at most one of $L$ or $R$ contains vertices that can be reached by lower-level tracks in the nested set.
\end{observation}

\begin{observation}
\label{obs:nest-higher}
If a semicircular track of $G_k$ at level $i$ is not the first or last track at its level, then it is nested inside a semicircular track at level $i+1$ or level $i+2$, respectively according to whether the given track is not subdivided or is subdivided by two junctions to higher-level tracks.
\end{observation}

\begin{lemma}
\label{lem:nested-domes}
If a yellow--blue edge of $G_k$ has  at least $x$ vertices of $G_k$ before it and after it in the sequence of $3^k$ vertices on the boundary line,
then there exists a set of $\Omega(\log x)$ pairwise-nested semicircular tracks directly above it, with levels that include either $i$ or $i+1$ for each $i\le\log_3 x$.
\end{lemma}

\begin{proof}
Construct the set of pairwise-nested tracks beginning with the track at level $0$ directly above the given edge, and then continue greedily, at each step choosing a track one or two levels above the previous choice by \cref{obs:nest-higher}. The track chosen at each step is necessarily nested with all the tracks nested within the previous outermost track, by transitivity of interval inclusion. This greedy construction can only stop at a level high enough for the outermost track to be the first or last semicircular track on its level, from which it follows that there are $\Omega(\log x)$ tracks, with levels that include either $i$ or $i+1$ for each $i\le\log_3 x$.
\end{proof}

\begin{theorem}
\label{thm:clique-width}
The graphs $G_k$ have unbounded clique-width.
\end{theorem}

\begin{proof}
We prove that, for any constant $r$, these graphs do not have rank-width $\le r$.
Assuming for a contradiction that they do all have rank-width $\le r$, they would have a balanced partition whose biadjacency matrix has low rank, by \cref{lem:balanced-cut}.
By \cref{lem:few-subintervals} we may assume that this partition forms $O(r)$ blocks in the vertex sequence of $G_k$.
By \cref{lem:close-intervals} we may assume that some two of these blocks $R$ and $B$, on different sides of the partition,
are larger than the gap $\ell$ between them by a factor of $3^{k/O(r^2)}$

Now choose any yellow--blue edge between the two blocks and apply \cref{lem:nested-domes} to find a nested sequence of semicircular tracks above this edge. Each track on level $i$ lies above a sub-drawing of a graph isomorphic to $G_i$, spanning a subsequence of $3^i$ vertices of the overall graph, of which it can reach $2^i$. Only $O(\log\ell)$ of these nested semicircular tracks can be at such a low level that they fail to reach both $R$ and $B$. Another $\Omega(\log 3^{k/O(r^2)})= \Omega_r(k)$ of them span subsequences of vertices that lie entirely within $R$, $B$, and the gap between them, reaching both $R$ and $B$.  For each one of these tracks, at level $i$ of the drawing, by \cref{obs:nest}, the leftmost and rightmost vertices reachable from the track give $r_i\in R$ and $b_i\in B$ that are connected to this track but not to any of the lower-level nested semicircular tracks. In particular, distinct $i$ and $j$ give distinct $r_i$, $r_j$, $b_i$, and~$b_j$.

The chosen vertices $r_i$ and $b_i$, for all of these  nested semicircular tracks, induce a biadjacency matrix in which each main diagonal coefficient, corresponding to the pair $(r_i,b_i)$, is a one. By \cref{obs:toofar}, for each such coefficient, either the coefficients above it in the same column of the matrix are all zero, or the coefficients to the left of it in the same row of the matrix are all zero, forming a matrix of the form described by \cref{lem:arrow}. By \cref{lem:arrow}, its rank equals the length of the sequence of semicircular tracks. But this is $ \Omega_r(k)$, unbounded, contradicting the assumption that the rank is $\le r$.
\end{proof}

\section{Bounded twin-width}

In this section, we prove that the strict outerconfluent graphs have bounded twin-width. The notion of \emph{planification} that we use, an encoding a strict outerconfluent graph and its drawing by a plane graph, is quite similar to the notion of a first-order transduction; transductions are known to preserve bounded twin-width \citep{TW1}. Rather than applying this general approach, we work out the details more carefully, using a counting argument,
to show that strict outerconfluent graphs have bounded twin-width as ordered graphs with a natural vertex ordering coming from their drawings.

\subsection{Small hereditary classes of ordered graphs}

As \citet{TW4} showed, for hereditary families of \emph{ordered} graphs, twin-width is intimately related to the growth rate of the class. Their \emph{small conjecture}, that the same relation held more strongly for unordered graphs, was later falsified \citep{TW7}. As it may be somewhat unfamiliar to treat ordered graphs as standalone objects, we provide definitions here.

\begin{definition}
An \emph{ordered graph} is a triple $G=(V,E,<)$ where $(V,E)$ are the sets of vertices and edges of an undirected graph and $(V,<)$ is a total ordering. Its \emph{number of vertices} is $|V|$ (this is commonly called the \emph{order} of a graph, to distinguish it from the \emph{size} $|E|$, but that would obviously cause confusion with respect to $<$, so we avoid this terminology.) Two ordered graphs are \emph{isomorphic} if there is a bijection on their vertex sets that is simultaneously a graph isomorphism on their undirected graphs and an order isomorphism on their vertex orders. Isomorphism is an equivalence relation; we call its equivalence classes \emph{isomorphism classes}. All members of an isomorphism class must have the same number of vertices, so we can talk about this number for an isomorphism class rather than an individual graph. An \emph{induced subgraph} $G[S]$ of an ordered graph $G=(V,E,<)$, defined by a subset $S\subseteq V$, is another ordered graph $(S,E_S,<_S)$ where $E_S$ is the subset of $E$ consisting of edges with both endpoints in $S$, and $<_S$ is the restriction of $<$ to $S$.
\end{definition}

We consider here only finite graphs: $V$ and $E$ must be finite sets. With that assumption, there can be only finitely many isomorphism classes that have a given number of vertices. However, we still speak about \emph{classes} of graphs, rather than \emph{sets} of graphs, to emphasize that we are not restricting the vertices of the graphs to belong to any specific universal set, such as the natural numbers or the points of the plane.

\begin{definition}
A class of ordered graphs is \emph{hereditary} if it contains every induced subgraph of a graph in the class. It is \emph{small} if there exists a number $c$ such that the number of isomorphism classes of $n$-vertex graphs in the class is $O(c^n)$.
\end{definition}

The following is central to our proof that strict outerconfluent graphs have bounded twin-width. Although it is one of the key results in the theory of twin-width, we are not aware of previous uses of this lemma to prove bounded twin-width of natural classes of graphs, rather than using more direct constructions.

\begin{lemma}[\citet{TW4}]
\label{lem:small-width}
Every small hereditary class of ordered graphs has bounded twin-width. A hereditary class of (unordered) undirected graphs has bounded twin-width if and only if its graphs can be ordered to form a small hereditary class of ordered graphs.
\end{lemma}

Additionally, we need an algorithmic version of this result:

\begin{lemma}[\citet{TW4}]
\label{lem:small-alg}
For every small hereditary class of ordered graphs there is a polynomial-time algorithm for constructing twin-width decompositions of bounded width for the ordered graphs in the class.
\end{lemma}

\subsection{Ordering outerconfluent graphs}

To apply \cref{lem:small-width} and \cref{lem:small-alg} to outerconfluent graphs, we need to describe these as ordered graphs, rather than merely as graphs. There is an obvious ordering to use for their vertices, the ordering around the boundary of the disk on which these graphs are drawn; this is a cyclic ordering rather than a linear ordering, but that is a mere technicality.

\begin{definition}
Define an \emph{ordered strict outerconfluent drawing} to be a strict outerconfluent drawing within a specified oriented topological disk, together with a choice of one vertex of the drawing to be the start of its linear order. For technical reasons we require all tracks and junctions of the drawing to be interior to the disk, rather than touching its boundary at non-vertex points as some tracks from \cref{fig:non-tree-like} do; this does not restrict the class of graphs that may be drawn in this way. As a special case we allow drawings with no vertices, tracks, or junctions, representing the empty graph, despite the inability of choosing a starting vertex in this case. The \emph{ordered graph} of an ordered strict outerconfluent drawing is the ordered graph $G=(V,E,<)$ for which $(G,E)$ is the undirected graph depicted in the drawing, and $<$ is the clockwise ordering of vertices around the boundary of the disk of the drawing, starting from the designated starting vertex. A \emph{strict outerconfluent ordered graph} is any graph that is the ordered graph of an ordered strict outerconfluent drawing. Two ordered strict outerconfluent drawings are \emph{topologically equivalent} if there is a smooth homeomorphism of the plane that maps each vertex, track and junction of one drawing to a corresponding vertex, track, or junction of the other drawing, preserving the clockwise orientation of the vertices around the disk and preserving the choice of starting vertices.
\end{definition}

\begin{observation}
\label{obs:top-equiv}
The ordered graphs of topologically equivalent ordered strict outerconfluent drawings are isomorphic ordered graphs.
\end{observation}

\begin{definition}
If $D$ is an ordered strict outerconfluent drawing, an \emph{induced subdrawing} $D[S]$, for a given subset $S$ of the vertices of the drawing, is obtained from $D$ by the following steps:
\begin{itemize}
\item Remove from $D$ all vertices that do not belong to $S$.
\item Remove from $D$ all tracks and junctions that do not belong to smooth curves connecting pairs of vertices in $S$.
\item While any remaining junction has exactly two tracks meeting at it (necessarily forming a locally smooth curve at that junction), remove the junction and replace the two tracks by their union, so that all remaining junctions form the meeting point of three or more tracks.
\item If $S$ is non-empty, select the starting vertex of the ordered induced drawing to be the vertex of $S$ that appears earliest in the vertex ordering of $D$.
\end{itemize}
\end{definition}

\begin{lemma}
\label{lem:induced-subdrawing}
If $G$ is the ordered graph of an ordered strict outerconfluent drawing $D$, and $S$ is any subset of the vertices of $G$, then the induced subgraph $G[S]$ is the ordered  graph of the induced subdrawing $D[S]$.
\end{lemma}

\begin{proof}
The removal of vertices, tracks, and junctions from $D$ leaves only the vertices in $G[S]$, and is defined in a way that does not change the existence of smooth curves between these vertices. The replacement of pairs of tracks by their union also does not affect the existence of smooth curves between pairs of vertices, as each such curve can only use both replaced tracks, or neither. The choice of starting vertex in the induced subdrawing is made in a way that causes the vertex ordering of the subdrawing to be the induced ordering of the vertex ordering of $D$.
Because $D$ is assumed to be strict, it has no multiple adjacencies between vertices, nor loops from a vertex to itself. Removing tracks from $D$ to form $D[S]$ cannot create new multiple adjacencies or loops, so this remains true in $D[S]$. Additionally, the removal of unused tracks and junctions from $D$ to form $D[S]$, and the merger of tracks at two-track junctions, ensures that in $D[S]$ the technical requirements of having no unused tracks or junctions, and having at least three tracks at each junction, are maintained.
\end{proof}

\begin{corollary}
\label{cor:hereditary}
The ordered strict outerconfluent graphs are a hereditary class of ordered graphs.
\end{corollary}

\begin{proof}
Each such graph has an ordered outerconfluent drawing representing it, and each of its induced subgraphs comes from the corresponding induced subdrawing by \cref{lem:induced-subdrawing}. Therefore, each induced subgraph has a drawing representing it, and remains in the same class of graphs.
\end{proof}

\subsection{Smallness}

To prove that the ordered strict outerconfluent graphs form a small class, we combine a general principle (beyond the scope of this paper to formalize) that the number of planar diagrams with a given number of features is singly exponential in the number of features, and a result from \citet{EppHolLof-JoCG-16} that certain strict confluent drawings have a linear number of features. We begin with a bound on the number of planar diagrams, as we need it here.

\begin{definition}
A \emph{plane graph} is a graph together with a non-crossing drawing in the plane: a point for each vertex, a smooth curve connecting the two endpoints of each edge, and no points of intersection between edge curves other than at their shared endpoints. Two plane graphs are \emph{topologically equivalent} if there is a homeomorphism of the plane mapping each feature of one to a corresponding feature of the other. A  plane graph is \emph{maximal} if it is not possible to add any more edges to the graph and corresponding edge curves to the drawing, connecting pairs of existing vertices that were not previously connected. A \emph{face} of a plane graph is a connected component of the topological space formed by removing the vertices and edge curves from the plane.
\end{definition}

The following facts are standard in topological graph theory:

\begin{lemma}
Every planar graph has a plane drawing. The maximal plane graphs on $n$ vertices, for $n\ge 3$, have exactly $3n-6$ edges and exactly $2n-4$ faces, each bounded by three edges of the graph. $2n-3$ of these faces are bounded, and one is unbounded. Their graphs, the \emph{maximal planar graphs}, each have exactly $4n-8$ equivalence classes of drawings, under topological equivalence, where an equivalence class is determined by the choice of which triangle in the graph is to be the outer face and how it is to be oriented.
\end{lemma}

\begin{lemma}[\citet{Tur-DAM-84}]
The number of isomorphism classes of planar graphs with $n$ vertices is at most $2^{12n}$.
\end{lemma}

\begin{corollary}
\label{cor:plane-small}
The number of equivalence classes of plane graphs, under topological equivalence, is at most $c^n$ for some $c>0$.
\end{corollary}

\begin{proof}
By Tur\'an's lemma, the number of maximal planar graphs is at most $2^{12n}$, from which it follows that the number of topological equivalence classes of maximal plane graphs is at most $(4n-8)2^{12n}$. Every plane graph can be obtained by removing some subset of the $3n-6$ edges of a maximal plane graph, so the number of topological equivalence classes of plane graphs is at most $(4n-8)2^{12n}2^{3n-6}$.
\end{proof}

The bounds stated above are far from tight, but this is unimportant for our results. To apply Tur\'an's lemma to strict outerconfluent drawings, we need to transform them into plane graphs.

\begin{definition}
\label{def:planification}
A \emph{face--vertex incidence} of a plane graph $G$ is a pair of a vertex of $G$ and a face of $G$ that has that vertex on its boundary.
Given an ordered strict outerconfluent drawing $D$, we define a \emph{planification} of $D$ to be a tuple $(G,o,s,S)$, where $G$ is a plane graph, $o$ and $s$ are vertices of $G$, and $S$ is a subset of the face-vertex incidences of $G$, constructed as follows:
\begin{itemize}
\item The vertices of $G$ consist of the vertices and junctions of $D$, and an additional vertex, $o$, placed outside the disk in which $D$ is drawn.
\item Vertex $s$ is the start vertex in the ordering of $D$.
\item The edge curves of $G$ consist of the tracks of $D$, together with additional curves, drawn outside the disk in which $D$ is drawn and disjoint from each other, connecting $o$ to each vertex of $D$. Two vertices are adjacent in $G$ exactly when they are connected by one of these curves. (Note that, in a strict outerconfluent drawing, it is impossible for two or more tracks to connect the same two vertices or junctions, as this would in all cases result in a loop or an unused track.)
\item $S$ consists of the \emph{sharp angles} of $D$, in the sense described by \citet{EppHolLof-JoCG-16}: they are the face--vertex incidences where the vertex of $G$ is a junction of $D$, at which the two tracks on either side of the face at that vertex do not have a smooth union. Necessarily, $S$ includes all but two of the vertex--face incidences at each junction. Two tracks at any junction have a smooth union if and only if, in the cyclic ordering of tracks at that junction, the two non-sharp angles separate the two tracks from each other.
\end{itemize}
\end{definition}

As we show, this completely encodes the combinatorial information in a strict outerconfluent drawing, in the following sense:

\begin{lemma}
\label{lem:planification}
Let $(G,o,s,S)$ be any planification of any ordered strict outerconfluent drawing $D$, and let $(G',o',s',S')$ be any tuple of a plane graph topologically equivalent to $G$, the vertices corresponding to $o$ and $s$ in $G'$ under the topological equivalence, and the subset of vertex--face incidences corresponding to $S$ under the topological equivalence. Then from $(G',o',s',S')$ we can construct a strict outerconfluent drawing $D'$ that is topologically equivalent to $D$.
\end{lemma}

\begin{proof}
Define a junction of $G'$ to be any of its vertices that is not a neighbor of $o$; these are the vertices that correspond to junctions of $D$ under the topological equivalence of $G$ and $G'$. Find disjoint neighborhoods of each junction, interior to the disk of the drawing; these exist because of the restriction that $D$ cannot touch the boundary of the disk except at its vertices. Within each neighborhood, modify $G'$ (preserving topological equivalence as a plane drawing) so that the edge-curves meeting at each junction form sharp or smooth angles according to the information given in $S'$. (Because of the technical restriction on our drawings, each junction has a neighborhood interior to the disk of the drawing, within which this modification may be performed.) Consider the result as a confluent drawing, with $o$ and its incident edges removed, neighbors of $o$ as its vertices, and non-neighbors of $o$ as junctions, embedded in a disk whose boundary lies in the faces of $G'$ incident with $o$. Order the vertices of the drawing clockwise around this disk starting with $s$. Then the result is a confluent drawing, mapped from drawing $D$ and its planification $G$ by a homeomorphism of the plane (the homeomorphism that maps $G$ to $G'$, composed with its local modifications at the junctions). This homeomorphism takes each vertex, track, or junction of $D$ to a corresponding vertex, track, or junction of the resulting confluent drawing, and preserves the smoothness or lack thereof of unions of tracks. Therefore, it is a topological equivalence of confluent drawings.
\end{proof}

Thus, we can count topological equivalence classes of strict outerconfluent drawings by using Tur\'an's lemma to count their planifications.
However, to apply Tur\'an's lemma, we need to know how many junctions and tracks there can be. Fortunately, this has already been bounded:

\begin{lemma}[\citet{EppHolLof-JoCG-16}]
\label{lem:few-junctions}
Every strict outerconfluent graph with $n$ vertices has a strict outerconfluent drawing with at most $n-3$ junctions and at most $3n-6$ tracks.
\end{lemma}

Putting these results together we have the main result of this section:

\begin{lemma}
\label{lem:small}
The ordered strict outerconfluent graphs form a small class of ordered graphs.
\end{lemma}

\begin{proof}
By \cref{lem:planification}, the number of equivalence classes of $n$-vertex ordered strict outerconfluent graphs under isomorphism of ordered graphs is at most the number of equivalence classes of planifications of drawings of these graphs, under topological equivalence of planifications. Given an $n$-vertex ordered strict outerconfluent graph, let $D$ be a strict outerconfluent drawing of it with at most $n-3$ junctions and at most $3n-6$ tracks, known to exist by \cref{lem:few-junctions}, and let $(G,o,s,S)$ be its planification. Each junction has as many vertex--face incidences as it has track--junction incidences; each track contributes two such incidences, one at each endpoint, but in the worst case each vertex of the given graph takes up at least one track endpoint (if it were an isolated vertex the number of junctions and tracks would be smaller) so the number of vertex--face incidences at junctions is at most $5n-12$. The number of vertices in $G$ is one plus the number of vertices and junctions in $D$, at most $2n-2$, so by \cref{cor:plane-small} the number of choices for the plane graph $G$ (under topological equivalence) is singly exponential in $n$. The number of choices for $o$ and $s$ is at most $2n-2$. The number of choices for $S$ is at most $2^{5n-12}$. Multiplying these numbers of choices together gives a singly exponential number of planifications, under topological equivalence, and therefore a singly-exponential number of ordered strict outerconfluent graphs, under ordered isomorphism.
\end{proof}

\subsection{Twin-width}

\begin{theorem}
The strict outerconfluent graphs have bounded twin-width. If the ordering of the vertices along the boundary of a strict outerconfluent drawing of one of these graphs is given, a twin-width decomposition for it of bounded width can be constructed in polynomial time.
\end{theorem}

\begin{proof}
This follows from \cref{cor:hereditary} and \cref{lem:small}, under which assigning them their boundary orderings produces a hereditary small class of ordered graphs,
\cref{lem:small-width}, under which  hereditary small class of ordered graphs have bounded twin-width, and 
\cref{lem:small-alg}, under which twin-width decompositions of bounded width for hereditary small class of ordered graphs can be found in polynomial time.
\end{proof}

\section{Discussion}

We have shown that the clique-width of strict outerconfluent graphs is unbounded, but our lower bound proves only sublogarithmic clique-width. It would be of interest to determine how quickly the clique-width can grow, as a function of the number of vertices. Many important graph optimization problems, including all problems expressible in the MSO$_1$ form of monadic second-order logic, can be solved efficiently for the graphs of bounded clique-width \citep{CouMakRot-TCS-00}, and our result presents an obstacle to the application of these techniques on strict outerconfluent graphs. If we could prove a superlogarithmic lower bound, it would cause algorithms whose dependence on clique-width is exponential to have a superpolynomial overall time bound, creating a greater barrier to their use.

In the other direction, we have shown that the twin-width of these graphs is bounded. This enables the application to these graphs of fixed-parameter-tractable algorithms for graphs of bounded treewidth, including algorithms for first-order model checking \citep{TW1} and for finding a graph coloring with a number of colors within a constant factor of the maximum clique size \citep{TW3}. However, because our proof goes through a counting argument, it does not provide a direct construction of a low-twin-width decomposition, and the bound that it provides on twin-width is large. It would be of interest to find an alternate proof with a better bound on twin-width.

It is natural to try to extend our twin-width bound to more general classes of confluent graphs.
The full class of all confluent graphs is out of reach, because it includes the interval graphs \citep{DicEppGoo-JGAA-05}, and these do not form a small class: even when counting the $n$-vertex interval graphs as unordered, undirected graphs their number is exponential in $n\log n$ \citep{YanPip-PAMSB-17}. We remark that this bound, together with our methods for converting counting problems on confluent drawings to plane graphs, can be used to show that some confluent drawings require $\Omega(n\log n)$ tracks; we omit the details. The strict confluent drawings produce a small class of unordered graphs by the same reasoning as \cref{lem:small}, but we do not know of a natural ordering for these graphs under which they are hereditary. On the other hand, the (non-strict) outerconfluent drawings naturally form a hereditary class of ordered graphs, by the same reasoning as \cref{cor:hereditary}, but we do not know whether they are small.

Many of the other known width parameters are either unbounded when clique-width is unbounded, or bounded when twin-width is bounded. Thus, our results settle whether these parameters are bounded on the strict outerconfluent graphs. However, it may be of interest to consider other width parameters in connection with other forms of confluent drawing.

\bibliographystyle{abbrvnat}
\bibliography{confluent-width}

\end{document}